\documentclass{amsart}

\usepackage{amsmath,amsfonts,amssymb,amsthm,graphics,amscd,enumerate,verbatim}
\usepackage{color, graphicx}
\usepackage{bm, pb-diagram}

\newtheorem {thm}{Theorem}[section]
\newtheorem{lem}[thm]{Lemma}
\newtheorem{prop}[thm]{Proposition}
\newtheorem{cor}[thm]{Corollary}
\newtheorem{defn}[thm]{Definition}

\newtheorem{ex}[thm]{Example}

\newtheorem{rem}[thm]{Remark}

\begin{document}

\title{Minus Partial Order in Regular modules}

\author{B. Ungor}
\address{Burcu Ungor, Department of Mathematics, Ankara University, Ankara, Turkey} \email{bungor@science.ankara.edu.tr}
\author{S. Halicioglu}
\address{Sait Hal\i c\i oglu, Department of Mathematics, Ankara University, Ankara, Turkey} \email{halici@ankara.edu.tr}
\author{A. Harmanci}
\address{Abdullah Harmanci, Department of Mathematics, Hacettepe University, Ankara, Turkey} \email{harmanci@hacettepe.edu.tr}
\author{J. Marovt}
\address{Janko Marovt, Faculty of Economics and Business, University of Maribor,
Razlagova 14, SI-2000 Maribor, Slovenia}
\email{janko.marovt@um.si}

\date{}

\begin{abstract} The minus partial order is already known for sets of
matrices over a field and bounded linear operators on
arbitrary Hilbert spaces. Recently, this partial order has
been studied on Rickart rings. In this paper, we extend the concept of the minus
relation to the module theoretic setting and prove that this
relation is a partial order when the module is regular. Moreover,
various characterizations of the minus partial order in regular
modules are presented and some well-known results are also
generalized.
\end{abstract}
\maketitle

\section{Introduction}

Let $\mathcal{S}$ be a semigroup and
$a\in\mathcal{S}$. Any solution $x=a^{-}$ to the equation $axa=a$
is called \textit{an inner generalized inverse }of $a$. If such
$a^{-}$ exists, then $a$ is called \textit{regular,} and if every
element in a semigroup $\mathcal{S}$ is regular, then
$\mathcal{S}$ is called \textit{a regular semigroup}. Hartwig
\cite{Hartwig} introduced \textit{the minus partial order} $\leq^{-}$ on
regular semigroups using generalized inverses. For a regular
semigroup $\mathcal{S}$ and $a,b\in\mathcal{S}$, we write
\begin{equation}
a\leq^{-}b\quad\text{if\quad}a^{-}a=a^{-}b\quad\text{and\quad}aa^{-}=ba^{-}
\label{HartwigMinusDef}%
\end{equation}
for some inner generalized inverse $a^{-}$ of $a$.

Let $B(\mathcal{H})$ be
the algebra of all bounded linear operators on a Hilbert space $\mathcal{H}$. For
an operator $A\in B(\mathcal{H})$, the symbols $\mathrm{Ker}$ $A$
and $\mathrm{Im}$ $A$  will denote the kernel and the image of $A$,
respectively. It is known that $A\in B(\mathcal{H})$ is regular if and only if
$\operatorname{Im}\,A=\overline {\operatorname{Im}\,A}$, i.e. the image of $A$ is
closed (see for example \cite{Nashed}).
\v{S}emrl studied in \cite{Semrl} the minus partial order on
$B(\mathcal{H})$. He did not want to restrict himself only to
operators on $B(\mathcal{H})$ with closed images so he defined a
new order $\leq_{S}$ on $B(\mathcal{H})$ in the following way: For
$A,B\in B(\mathcal{H})$ we write $A\leq_{S}B$ if there exist
idempotent operators $P,Q\in B(\mathcal{H})$ such that
$\operatorname{Im}\,P=\overline {\operatorname{Im}\,A}$,
$\mathrm{Ker}\,A=\mathrm{Ker}\,Q$, $PA=PB$, and
$AQ=BQ$. \v{S}emrl called this order the minus partial order on $B(\mathcal{H}%
)$ and proved that this is indeed a partial order on
$B(\mathcal{H})$ for a general Hilbert space $\mathcal{H}$. He
also showed that the partial order $\leq_{S}$ is the same as
Hartwig's minus partial order $\leq^{-}$ when $\mathcal{H}$ is
finite dimensional.

For a subset $\mathcal{A}$ of a ring $R$, $l_{R}(\mathcal{A})$ and $r_{R}(\mathcal{A})$ will
denote the left annihilator and the right annihilator of $\mathcal{A}$ in
$R$, respectively. If the subset $\mathcal{A}$ is a singleton, say
$\mathcal{A}=\{a\}$, then again we simply write $l_{R}(a)$ and $r_{R}(a)$,
respectively. A ring $R$ is called a Rickart ring if for every $a\in R$
there exist idempotent elements $p,q\in R$ such that $r_{R}(a)=p\cdot R$
and $l_{R}(a)=R\cdot q$. Note that every
Rickart ring $R$ has the (multiplicative) identity and that the class of
Rickart rings includes von Neumann algebras and rings with no proper zero divisors
 (see \cite{Baer}
or \cite{K}).

Following \v{S}emrl's approach, the authors further generalized in
\cite{DjordjevicRakicMarovt} the minus partial order to Rickart
rings. A new relation was introduced on a ring with identity: Let
$R$ be a ring with the identity $1_R$ and $a,b\in R$. Then we
write $a\leq^{-}b$ if there exist idempotent elements $p,q\in R$
such that
\begin{equation}
l_{R}(a)=R(1_R-p),  \quad r_{R}(a)=(1_R-q)R, \quad pa=pb, \quad \text{and} \quad aq=bq.
\label{def_minus_Rickart}
\end{equation}
It was proved that this relation $\leq^{-}$ is indeed a partial
order when $R$ is a Rickart ring and that definitions
(\ref{HartwigMinusDef}) and (\ref{def_minus_Rickart}) are
equivalent when $R$ is a ring in which every element is regular,
i.e. $R$ is a \textit{von Neumann regular ring}.

In this paper, we will study the minus partial order in a more
general setting. The goal of this paper is to extend the notion of
the minus partial order to modules using their endomorphism rings.
We show that this minus relation is a partial order when the
module is  regular. We also present various characterizations of
the minus partial order in regular modules and generalize some
well-known results.

Throughout this paper $R$ denotes an associative ring
with identity $1_R$ and modules are unitary right $R$-modules. For a
right $R$-module $M_R=M$, $S =$ End$_R(M)$ is the ring of all right $R$-module
endomorphisms of $M$. It is well known that $M$ is a left $S$ and
right $R$-bimodule. In this work, for the $(S, R)$-bimodule $M$,
$l_S(.)$ and $r_R(.)$ stand for the left annihilator of a subset
of $M$ in $S$ and the right annihilator of a subset of $M$ in $R$,
respectively. If the subset is a singleton, say $\{m\}$, then we
simply write $l_{S}(m)$ and $r_{R}(m)$, respectively.

\section{Minus partial order in modules}

Let $M$ be a right $R$-module with $S =$ End$_R(M)$. For the sake
of brevity, in the sequel, $S$ will stand for the endomorphism
ring of the module $M$ considered. We will denote the identity map
in $S$ by $1_S$. An element $m \in M$ is called ({\it
Zelmanowitz}) {\it regular} if $$m=m\varphi (m) \equiv
\hspace{0,1cm} m\varphi m$$ for some $\varphi\in M^*$ where $M^*
=$ Hom$_R(M, R)$ denotes the dual of $M$.

For a ring $R$, let $a\in R$ be a regular element (in the sense of
von Neumann) so that there exists $b \in R$ such that $a=aba$.
Define the map $\varphi:R\rightarrow R$ with $\varphi (r)=br, r\in
R$. Then $\varphi \in R^* =$ End$_R(R)$. We have $$a\varphi
a=a\varphi(a)=aba=a$$ which yields that $a$ is regular in $R_R$
(in the sense of Zelmanowitz, see also \cite{ZEL}).

Let now $R$ be a ring with identity and suppose $a$ is a regular element in $R_R$ (in the sense of Zelmanowitz).
Then there exists  $\varphi\in R^* $ such that $a=a\varphi a$.
Define $a^-=\varphi(1_R)\in R$. Then  $$a=a\varphi (a)=a\varphi(1_Ra)=a\varphi(1_R)a=aa^-a.$$
We may conclude that $a \in R$ is regular if and only if $a$ is regular in
$R_R$ (or, similarly, in the left-$R$ module $_RR$). A module $M$ is called \textit{regular} (in the
sense of Zelmanowitz) if every element of $M$ is
regular.

\begin{rem}\label{decomposition} \cite[Lemma B.47]{NY} {\rm Let $M$ be a module and $m\in M$ regular,
say $m=m\varphi m$ where $\varphi\in M^*$. Then $e=\varphi m\in R$
is an idempotent, $mR\cong eR$ is projective, and $M=mR\oplus N$
where $N=\{n\in M \mid m\varphi n=0\}$. }\end{rem}

\begin{rem}\label{burcu} {\rm Let $M$ be a module. It is known that Hom$_R(R,M)\cong M$.
Let $m\in M$ be regular, say $m=m\varphi m$ where $\varphi\in
M^*$. Then for the map $m\varphi:M\rightarrow M$ defined by $$(m\varphi)(x)=m\varphi (x) \equiv m\varphi x, \quad x\in M$$
we may conclude that $m\varphi \in S$ and that $m \varphi$ is an idempotent in $S$.
}\end{rem}

With the following definition we will introduce the concept of
minus order in the setting of modules.

\begin{defn}\label{ilk} {\rm Let $M$ be a module and $m_1, m_2\in M$.
We write $m_1\leq^{-} m_2$
if there exists $\varphi \in M^*$ such that $m_1=m_1\varphi m_1$,
$m_1\varphi=m_2 \varphi$, and $\varphi m_1=\varphi m_2$. We call
the relation $\leq^{-}$ the \textit{minus order} on $M$.}
\end{defn}

We will prove that when the module $M$ is regular, the relation $\leq^{-}$ is a partial order. First, let us present an auxiliary result and a new characterization of the minus order.

\begin{prop}\label{subset} Let $M$ be a module and $m_1, m_2\in M$.
If $m_1\leq^{-} m_2$, then $m_1R\subseteq m_2R$.
\end{prop}

\begin{proof} Assume that $m_1\leq^{-} m_2$. Then there exists $\varphi \in M^*$
such that $m_1=m_1\varphi m_1$,
$m_1\varphi=m_2 \varphi$, and $\varphi m_1=\varphi m_2$. Hence
$m_1=m_1\varphi m_1=m_2\varphi m_1=m_2\varphi m_2 \in m_2R$ and thus
$m_1R\subseteq m_2R$.
\end{proof}

\begin{thm}\label{main} Let $M$ be a module and $m_1, m_2\in M$
with $m_1$ regular. Then $m_1\leq^{-} m_2$ if and only if there
exist $f^2=f\in S$, $a^2=a\in R$ such that $l_S(m_1)=l_S(f)$,
$r_R(m_1)=r_R(a)$, $fm_1=fm_2$, and $m_1a=m_2a$.
\end{thm}
\begin{proof} Let $m_1\leq^{-} m_2$. Then there exists $\varphi \in M^*$ such that $m_1=m_1\varphi m_1$,
$m_1\varphi=m_2 \varphi$, and $\varphi m_1=\varphi m_2$. Let
$f=m_1\varphi$ and $a=\varphi m_1$. Then $f^2=f\in S$ and
$a^2=a\in R$. Clearly, $l_S(m_1)\subseteq l_S(f)$. From
$m_1=m_1\varphi m_1$ we obtain $(1_S-m_1\varphi)m_1=0$. So,
$l_S(f)=S(1_S-m_1\varphi)\subseteq l_S(m_1)$. Hence
$l_S(m_1)=l_S(f)$. Obviously, $r_R(m_1)\subseteq r_R(a)$.
Similarly, since $m_1=m_1\varphi m_1$, $m_1(1_R-\varphi m_1)=0$, it follows
$r_R(a)=(1_R-\varphi m_1)R\subseteq r_R(m_1)$. Thus
$r_R(m_1)=r_R(a)$. Also, $fm_1=m_1\varphi
m_1=m_1\varphi m_2=fm_2$ and $m_1a=m_1\varphi m_1=m_2\varphi m_1
=m_2a$.

Conversely, assume that there exist $f^2=f\in S$, $a^2=a\in R$
such that $l_S(m_1)=l_S(f)$, $r_R(m_1)=r_R(a)$, $fm_1=fm_2$, and
$m_1a=m_2a$. Also, $m_1=m_1\varphi m_1 $ for some $\varphi\in
M^*$. Then $m_1=fm_1=m_1a$. Let $\beta=a\varphi f\in M^*$. Hence
$m_1\beta m_1 =m_1a\varphi  fm_1 =m_1\varphi m_1=m_1$. Also, we
have $m_1\beta=m_1a\varphi f=m_2a\varphi f=m_2 \beta$ and $\beta
m_1 =a\varphi  fm_1=a\varphi fm_2=\beta  m_2$. Therefore
$m_1\leq^{-} m_2$.
\end{proof}

We are now in position to prove that the minus order $\leq^{-}$ is indeed
a partial order when the module $M$ is regular.

\begin{thm}\label{ant} Let $M$ be a regular module.
Then the relation  $\leq^{-} $ is a partial order on $M$.
\end{thm}
\begin{proof} \textit{Reflexivity}: Obvious.\\
\textit{Antisymmetry}: Let $m_1, m_2\in M$ with $m_1\leq^{-} m_2$
and $m_2\leq^{-} m_1$. Then there exist $\varphi, \beta \in M^*$
such that
$$m_1=m_1\varphi m_1,
m_1\varphi=m_2 \varphi, \varphi m_1=\varphi m_2 $$ and
$$m_2=m_2\beta m_2,
m_2\beta=m_1 \beta, \beta m_2=\beta m_1.$$ By Proposition
\ref{subset}, $m_1R=m_2R$.  Remark \ref{decomposition} yields that
$M=m_1R\oplus N_1$ where $N_1=\{n\in M\mid m_1\varphi n=0\}$.
Since $m_1\varphi(m_1-m_2)=0$, $m_1-m_2\in N_1$, say $m_1-m_2=n\in
N_1$. Hence $m_1=n+m_2$, and also $m_2=m_1r$ for some $r\in R$.
 Thus $m_1=n+m_1r$, so $n=m_1-m_1r\in N_1\cap m_1R=\{0\}$.
 This implies that $m_1=m_1r$. Therefore $m_1=m_2$.\\
\textit{Transitivity}: Let $m_1, m_2, m_3\in M$ with $m_1\leq^{-}
m_2$ and $m_2\leq^{-} m_3$. By Theorem \ref{main}, there exist
$f_1^2=f_1\in S$, $a_1^2=a_1\in R$ such that $l_S(m_1)=l_S(f_1)$,
$r_R(m_1)=r_R(a_1)$, $f_1m_1=f_1m_2$, $m_1a_1=m_2a_1$; and
$f_2^2=f_2\in S$, $a_2^2=a_2\in R$ such that $l_S(m_2)=l_S(f_2)$,
$r_R(m_2)=r_R(a_2)$, $f_2m_2=f_2m_3$, $m_2a_2=m_3a_2$. From
$(1_S-f_1)m_1=0$ and $m_1(1_R-a_1)=0$ it follows
\begin{equation}\label{eqJMB}
m_1=f_1m_1=f_1m_2=m_1a_1=m_2a_1.
\end{equation}
Let $f=f_1+f_1f_2(1_S-f_1)$ and
$a=a_1+(1_R-a_1)a_2a_1$. Then $f^2=f\in S$ and $a^2=a\in R$. We
claim that \begin{center} $l_S(m_1)=l_S(f)$, $r_R(m_1)=r_R(a)$,
$fm_1=fm_3$, and $m_1a=m_3a$.
\end{center}

\begin{enumerate}
    \item[(i)] Clearly, $l_S(m_1)=l_S(f_1)\subseteq l_S(f)$. Let $g\in
l_S(f)$. Then $gff_1=0$. It follows $gf_1=0$ and so $g\in
l_S(f_1)=l_S(m_1)$.
 Hence $l_S(m_1)=l_S(f)$.
    \item[(ii)] It is obvious that $r_R(m_1)=r_R(a_1)\subseteq r_R(a)$.
Let $x\in r_R(a)$. Then $a_1ax=0$. This yields $a_1x=0$, i.e.
$x\in r_R(a_1)=r_R(m_1)$, and thus $r_R(m_1)=r_R(a)$.
    \item[(iii)] We will now prove $fm_1=fm_3$. By (\ref{eqJMB}), $fm_1=m_1$.
    Similarly, $fm_2=m_2$.
 From the definition of $f$ we obtain, $fm_3=f_1m_3+f_1f_2m_3-f_1f_2f_1m_3$.
 Note that $f_1f_2m_3=f_1f_2m_2=f_1m_2=f_1m_1=m_1$.
 Then $$fm_3=m_1+f_1m_3-f_1f_2f_1m_3=m_1+f_1(1_S-f_2)f_1m_3.$$
 We claim that $f_1(1_S-f_2)f_1=0$, that is, $f_1(1-f_2)\in l_S(f_1)=l_S(m_1)$.
 By $f_1f_2m_1=f_1f_2m_2a_1=f_1m_2a_1=f_1m_1$, we obtain $f_1(1_S-f_2)m_1=0$ and
 hence $f_1(1_S-f_2)f_1=0$, as claimed. Thus $fm_3=m_1$ and
 therefore $fm_1=fm_3$.
    \item[(iv)] We assert that $m_1a=m_3a$. We have $m_1a_1=m_1$ due to
$m_1(1_R-a_1)=0$. Similarly, $m_2a_2=m_2$. From
$m_3a=m_3a_1+m_3a_2a_1-m_3a_1a_2a_1$ and since
$m_3a_2a_1=m_2a_2a_1=m_2a_1=m_1a_1=m_1$, we obtain
$$m_3a=m_1+m_3a_1(1_R-a_2)a_1.$$ We now claim that $a_1(1_R-a_2)a_1=0$, i.e.
$(1_R-a_2)a_1\in r_R(a_1)=r_R(m_1)$. By
$m_1a_2a_1=f_1m_2a_2a_1=f_1m_2a_1=m_1a_1$ it follows that
$m_1(1_R-a_2)a_1=0$. Then $a_1(1_R-a_2)a_1=0$, as required. Hence
$m_3a=m_1$ and thus $m_1a=m_3a$.
\end{enumerate}
By Theorem \ref{main}, $m_1\leq^{-} m_3$ and so the relation
$\leq^{-} $ is a partial order on $M$.
\end{proof}

We close this section by giving the star analogues of the minus
order on modules. A ring equipped with an involution $\ast$ will
be called a \textit{$\ast$-ring}. A $\ast$-ring $R$ is called a
\textit{Rickart $\ast$-ring} if the left annihilator $l_{R}(a)$ of
any element $a\in R$ is generated by a projection (i.e. a
self-adjoint idempotent), equivalently, the right annihilator
$r_{R}(a)$ of any element $a\in R$ is generated by a projection.
Recall that Drazin \cite{Drazin} defined the star partial order in
a general setting of proper $\ast$-semigroups and note that
natural special cases of proper $\ast $-semigroups are all
\textit{proper $\ast$-rings}, with ``properness" defined via
$aa^{\ast}=0$ implying $a=0$. Recall also (see e.g.  \cite{Baer})
that any Rickart $\ast$-ring has the identity and is a proper
$\ast$-ring.

For  Rickart $\ast$-rings, the following characterization of the
star partial order $\underset{\ast}{\leq}$ was given in
\cite{MarovtRakicDjodjevic}. Let $R$ be a Rickart $\ast$-ring. For
$a,b\in R$ we have $a\underset{\ast}{\leq}b$ if and only if there
exist projections $p,q\in R$ such that
\begin{equation}
l_{R}(a)=R(1_R-p),\;r_{R}(a)=(1_R-q)R,\;pa=pb,\;\text{and
}\;aq=bq. \label{Rakic}
\end{equation}
Observe here (see e.g. \cite[Lemma 2.1]{DjordjevicRakicMarovt})
that for any idempotent $p$ in a ring $R$ with identity,
\begin{equation}\label{prop_idempotents}
R(1_R-p)=l_R(p) \quad \text{and} \quad
 (1_R-p)R=r_R(p).
 \end{equation}

Two orders that are closely related to the star and the minus
partial orders are \textit{the left-star} and \textit{the
right-star partial orders}. These orders were introduced by
Baksalary and Mitra in \cite{BaksalaryMitra} on the set of
$m\times n$ complex matrices $M_{m,n}(\mathbb{C})$. The left-star
partial order is defined as follows. For $A,B\in
M_{m,n}(\mathbb{C})$,
$$
A\ast\!\!\leq B\text{\quad if\quad}A^{\ast}A=A^{\ast}B\text{ and }
\mathrm{Im}\,A\subseteq\mathrm{Im}\,B.
$$
The right star partial order ${\leq\hspace{-0.1cm}\ast}$ is
defined symmetrically:\ $A{\leq\hspace{-0.1cm}\ast}\,B$ if and
only if $AA^{\ast}=BA^{\ast}$ and
$\mathrm{Im}\,A^{\ast}\subseteq\mathrm{Im}\,B^{\ast}$. These two
orders were generalized in \cite{DolinarGutermanMarovt} from the
set of all $n\times n$ complex matrices to $B(\mathcal{H})$ where
$\mathcal{H}$ is an arbitrary complex Hilbert space. Using
annihilators, authors further generalized in
\cite{MarovtRakicDjodjevic} the left and the right-star partial
orders to Rickart $\ast$-rings.

Motivated (\ref{Rakic}) and (\ref{prop_idempotents}), and by
generalizations of the left-star and the right-star orders to
Rickart $\ast$-rings in \cite{MarovtRakicDjodjevic}, we will now
extend the notion of these orders to the module theoretic setting.

\begin{defn}\label{ikinci}{\rm Let $M$ be a module and $m_1, m_2\in
M$. We write
\begin{itemize}
\item[(i)] $m_1\leq\hspace{-0.1cm}\ast~ m_2$  if there
exist $f^2=f\in S$, $a^2=a=a^*\in R$ such that $l_S(m_1)=l_S(f)$,
$r_R(m_1)=r_R(a)$, $fm_1=fm_2$, and $m_1a=m_2a$, where $R$ is a
 $\ast$-ring. We call the relation
$\leq\hspace{-0.1cm}\ast$ the right-star order on $M$.

\item[(ii)] $m_1~\ast\hspace{-0.1cm}\leq~ m_2$  if there
exist $f^2=f=f^*\in S$, $a^2=a\in R$ such that $l_S(m_1)=l_S(f)$,
$r_R(m_1)=r_R(a)$, $fm_1=fm_2$, and $m_1a=m_2a$, where $S$ is a
 $\ast$-ring. We call the relation
$\ast\hspace{-0.1cm}\leq$ the left-star order on $M$.

\item[(iii)] $m_1\underset{\ast}{\leq} m_2$  if there
exist $f^2=f=f^*\in S$, $a^2=a=a^*\in R$ such that
$l_S(m_1)=l_S(f)$, $r_R(m_1)=r_R(a)$, $fm_1=fm_2$ and $m_1a=m_2a$,
where $R$, and $S$ are $\ast$-rings. We call the relation
$\underset{\ast}{\leq}$ the star order on $M$.
\end{itemize}
}\end{defn}

We will prove that the relations introduced with Definition
\ref{ikinci} are partial orders when the module $M$ over a Rickart
$*$-ring is regular.

\begin{thm}\label{savci} Let $M$ be a regular module. Then the
following hold.
\begin{enumerate}
    \item[{\rm(1)}] If $R$ is a Rickart $\ast$-ring, then the relation
    $\leq\hspace{-0.1cm}\ast$, introduced
with Definition \ref{ikinci}, is a partial order on $M$.
    \item[{\rm(2)}] If $S$ is a Rickart $\ast$-ring, then the relation
    $\ast\hspace{-0.1cm}\leq$, introduced
with Definition \ref{ikinci}, is a partial order on $M$.
    \item[{\rm(3)}] If $R$ and $S$ are Rickart $\ast$-rings, then the
    relation $\underset{\ast}{\leq}$, introduced
with Definition \ref{ikinci}, is a partial order on $M$.
\end{enumerate}
\end{thm}

\begin{proof} We will only prove (1). The proofs of (2) and (3) are similar. \\
 \textit{Reflexivity:} Let $m\in M$. Since $M$ is regular, there exists
 $\varphi \in M^*$ such that $m=m\varphi m$. Let
$f=m\varphi$ and $b=\varphi m$. Then $f^2=f\in S$ and $b^2=b\in
R$. Clearly, $l_S(m)\subseteq l_S(f)$. From $m=m\varphi m$ we have
$m=fm$. Thus $l_S(f)\subseteq l_S(m)$ and hence $l_S(m)=l_S(f)$.
Similarly, $r_R(m)=r_R(b)$.  By assumption $R$ is a Rickart
$\ast$-ring so since $b\in R$, there exists (the unique) $a\in R$
with $a=a^2=a^{\ast}$ and $r_R(a)=r_R(b)$. To sum up, there exists
$f^2=f\in S$, $a^2=a=a^*\in R$ such that $l_S(m)=l_S(f)$,
$r_R(m)=r_R(a)$.  We may conclude that $m\leq\hspace{-0.1cm}\ast~ m$
for every $m\in M$.\\
\textit{Antisymmetry:} Let $m_1, m_2\in M$ with
$m_1\leq\hspace{-0.1cm}\ast~ m_2$ and $m_2\leq\hspace{-0.1cm}\ast~
m_1$. Then there exist $f_1^2=f_1\in S$, $a_1^2=a_1=a_1^*\in R$
such that $l_S(m_1)=l_S(f_1)$, $r_R(m_1)=r_R(a_1)$,
$f_1m_1=f_1m_2$, $m_1a_1=m_2a_1$; and $f_2^2=f_2\in S$,
$a_2^2=a_2=a_2^*\in R$ such that $l_S(m_2)=l_S(f_2)$,
$r_R(m_2)=r_R(a_2)$, $f_2m_2=f_2m_1$, $m_2a_2=m_1a_2$. By
(\ref{prop_idempotents}), $l_S(f_1)=S(1_S-f_1)$ and
$r_R(a_1)=(1_R-a_1)R$ which imply by assumption that
$(1_S-f_1)m_1=0$ and $m_1(1_R-a_1)=0$. So,
$m_1=f_1m_1=m_1a_1=m_2a_1$. Similarly, $m_2=f_2m_2=f_2m_1$. Hence
$$m_2=f_2m_1=f_2m_2a_1=m_2a_1=m_1.$$
\textit{Transitivity:} Let $m_1, m_2, m_3\in M$ with
$m_1\leq\hspace{-0.1cm}\ast~ m_2$ and $m_2\leq\hspace{-0.1cm}\ast~
m_3$. Then there exist $f_1^2=f_1\in S$, $a_1^2=a_1=a_1^*\in R$
such that $l_S(m_1)=l_S(f_1)$, $r_R(m_1)=r_R(a_1)$,
$f_1m_1=f_1m_2$, $m_1a_1=m_2a_1$; and $f_2^2=f_2\in S$,
$a_2^2=a_2=a_2^*\in R$ such that $l_S(m_2)=l_S(f_2)$,
$r_R(m_2)=r_R(a_2)$, $f_2m_2=f_2m_3$, $m_2a_2=m_3a_2$. From
$(1_S-f_1)m_1=0$ and $m_1(1_R-a_1)=0$ we obtain
$$m_1=f_1m_1=f_1m_2=m_1a_1=m_2a_1.$$ Let $f=f_1f_2$ and $a=a_1$. To
conclude the proof we will show that $f^2=f\in S$ and \begin{center}
$l_S(m_1)=l_S(f)$,
$fm_1=fm_3$, and $m_1a=m_3a$.
\end{center}
From  $r_R(m_2)=r_R(a_2)$ we have $m_2=m_2a_2$ and thus
$m_1a_1a_2=f_1m_2a_2=f_1m_2=m_1$. So $m_1(a_1a_2-1_R)=0$. Since
$r_R(m_1)=r_R(a_1)$, it follows that $a_1(a_1a_2-1_R)=0$ and hence
$a_1a_2=a_1$.  Recall that $a_1$ and $a_2$ are self-adjoined. So,
$a_1=a_2a_1$ which yields
$$m_1a=m_1a_1=m_2a_1=m_2a_2a_1=m_3a_2a_1=m_3a_1=m_3a.$$
Since $l_S(m_2)=l_S(f_2)$, we have $m_2=f_2m_2$ which implies
$f_2m_1=f_2m_2a_1=m_2a_1=m_1$. So,
$(1_S-f_1f_2)m_1=m_1-f_1f_2m_1=m_1-f_1m_1=m_1-m_1=0$. Since
$l_S(m_1)=l_S(f_1)$, $(1-f_1f_2)f_1=0$, so $f_1=f_1f_2f_1$. Now
$fm_1=f_1f_2m_1=m_1$ and $fm_3=f_1f_2m_3=f_1f_2m_2=f_1m_2=m_1$. So
$$fm_1=fm_3.$$ Also $f^2=(f_1f_2)^2=f_1f_2f_1f_2=f_1f_2$. Finally,
since $f=f_1f_2$ and $f_1=f_1f_2f_1=ff_1$ we get,
$l_S(m_1)=l_S(f_1)=l_S(f)$.
 We may conclude that $m_1\leq\hspace{-0.1cm}\ast~ m_3$.
\end{proof}

\section{Some Characterizations of the Minus order on Modules}

In this section we investigate some properties and
characterizations of the minus partial order on modules.

\begin{prop}\label{savci} {\rm Let $M$ be a regular module and
$m_1, m_2\in M$.  Then we have the following.
\begin{enumerate}
    \item $m_1\leq^{-} m_2$ if and only if $g m_1\leq^{-} g m_2$
    for every invertible element $g \in S$.
     \item $m_1\leq^{-} m_2$ if and only if $m_1 b\leq^{-} m_2 b$
    for every invertible element $b \in R$.
\end{enumerate}}\end{prop}

\begin{proof} (1) Assume first that $m_1\leq^{-} m_2$. By Theorem \ref{main}, there exist
$f^2=f\in S$, $a^2=a\in R$ such that $l_S(m_1)=l_S(f)$,
$r_R(m_1)=r_R(a)$, $fm_1=fm_2$, and $m_1a=m_2a$. Let $g \in S$ be invertible. Note first that $gm_1$ is a regular element in $M$ since $M$ is a regular module. Denote $p=gfg^{-1}$. Then $p \in S$ is an idempotent. We will show that $l_S(g m_1)=l_S(p)$,
$r_R(gm_1)=r_R(a)$, $pgm_1=pgm_2$, and $gm_1a=gm_2a$.

\begin{enumerate}
    \item [(i)] If $h \in l_S(g m_1)$, then $hgm_1=0$ and so $hgfm_1=0$. Hence
$hgf \in l_S(m_1)=l_S(f)$. Thus $hgf=0$ and so $hgfg^{-1}=0$.
Therefore $l_S(g m_1)\subseteq l_S(p)$. Conversely, if $x \in
l_S(p)$, then $xp=xgfg^{-1}=0$. Hence
$xgm_1=xgfm_1=xgfg^{-1}gm_1=0$ which yields $l_S(p) \subseteq l_S(g
m_1)$ and therefore $l_S(g m_1)=l_S(p)$.
    \item [(ii)] If $x \in r_R(gm_1)$, then $gm_1x=0$ and so
    $m_1x=0$. Hence $x \in r_R(m_1)=r_R(a)$. Thus $r_R(gm_1)\subseteq
    r_R(a)$. Conversely, if $x \in r_R(a)$, then $m_1x=0$ and so
    $gm_1x=0$. Hence $x \in r_R(gm_1)$ and thus $r_R(a)\subseteq
    r_R(gm_1)$. So, $r_R(gm_1)=r_R(a)$.
    \item [(iii)] $pgm_1=gfg^{-1}gm_1=gfm_1=gfm_2=gfg^{-1}gm_2=pgm_2$.
    \item [(iv)] Since $m_1a=m_2a$, $gm_1a=gm_2a$.
\end{enumerate}
Therefore $g m_1\leq^{-} g m_2$. \\
If  $g m_1\leq^{-} g m_2$ for an invertible $g\in S$, then $g^{-1} g m_1\leq^{-} g^{-1}g m_2$, i.e. $m_1\leq^{-} m_2$. \\
(2) Assume first, that
$m_1\leq^{-} m_2$. By Theorem \ref{main}, there exist $f^2=f\in
S$, $a^2=a\in R$ such that $l_S(m_1)=l_S(f)$, $r_R(m_1)=r_R(a)$,
$fm_1=fm_2$, and $m_1a=m_2a$. For an invertible element $b \in R$,
let $q=b^{-1}ab$.  Note again that $m_1b$ is a regular element in $M$ since $M$ is a regular module.  Then $q \in R$ is an idempotent. We will show
that $l_S(m_1b)=l_S(f)$, $r_R(m_1b)=r_R(q)$, $m_1bq=m_2bq$, and
$fm_1b=fm_2b$.
\begin{enumerate}
    \item [(i)] If $h \in l_S(m_1b)$, then $hm_1b=0$ and so $hm_1bb^{-1}=0$.
    Hence $hm_1=0$ and so $h \in l_S(m_1)=l_S(f)$. Thus
     $l_S(m_1b)\subseteq l_S(f)$. Conversely, if $h \in
l_S(f)=l_S(m_1)$, then $hm_1=0$ and so $hm_1b=0$. Hence $h \in
l_S(m_1b)$, i.e. $l_S(f) \subseteq l_S(m_1b)$, and therefore
$l_S(m_1b)=l_S(f)$.
    \item [(ii)] If $x \in r_R(q)$, then $qx=0$ and thus
    $b^{-1}abx=0$. Hence $bb^{-1}abx=0$ and so $abx=0$. From
    $bx \in r_R(a)=r_R(m_1)$ we obtain $m_1bx=0$ and therefore $r_R(q)\subseteq
    r_R(m_1b)$. Conversely, if $x \in r_R(m_1b)$, then $m_1bx=0$ which yields
    $bx \in r_R(m_1)=r_R(a)$.  Hence $abx=0$ and so $b^{-1}abx=0$.
    Thus $qx=0$, i.e.  $x \in r_R(q)$. It follows that $r_R(m_1b)\subseteq
    r_R(q)$ and therefore $r_R(m_1b)=r_R(q)$.
    \item [(iii)] $m_1bq=m_1bb^{-1}ab=m_1ab=m_2ab=m_2bb^{-1}ab=m_2bq$.
    \item [(iv)] Since $fm_1=fm_2$, $fm_1b=fm_2b$.
\end{enumerate}
We conclude that $m_1 b\leq^{-} m_2 b$. \\
 If  $m_1 b\leq^{-} m_2 b$ for an invertible $b\in R$, then $m_1b b^{-1}\leq^{-} m_2 b b^{-1}$, i.e. $m_1\leq^{-} m_2$.
\end{proof}

We will next prove that if we somewhat relax conditions
on left and right annihilators in Theorem \ref{main}, we still
obtain the minus partial order on regular modules.

\begin{thm}\label{elif} Let $M$ be a regular module and $m_1, m_2\in M$.
Then $m_1\leq^{-} m_2$ if and only if there exist $f^2=f\in S$,
$a^2=a\in R$ such that $l_S(f)\subseteq l_S(m_1)$, $r_R(a)
\subseteq r_R(m_1)$, $fm_1=fm_2$, and $m_1a=m_2a$.
\end{thm}

\begin{proof}
Suppose first that there exist $f^2=f\in S$, $a^2=a\in R$ such
that $l_S(f)\subseteq l_S(m_1)$, $r_R(a) \subseteq r_R(m_1)$,
$fm_1=fm_2$, and $m_1a=m_2a$. Since $M$ is a regular module,
for $m_1\in M$, there exists $\varphi\in M^*$ such that $m_1
=m_1\varphi m_1$. Let $$f_1=m_1\varphi \in S \quad
\text{and}\quad a_1=\varphi m_1 \in R.$$ Clearly, $f_1^2
=f_1$ and $a_1^2=a_1$. We will first show that
$l_{S}(f_1)=l_{S}(m_1)$ and $r_{R}(a_1)=r_{R}(m_1)$. Let $g \in
l_{S}(m_1)$. Then $g m_1=0$ and so $gf_1 = g m_1\varphi =0$. Hence
$g \in l_{S}(f_1)$ and thus $l_{S}(m_1) \subseteq l_{S}(f_1)$.
Conversely, if $g \in l_{S}(f_1)$,  then $g f_1=0$ and so $gm_1 =
g m_1\varphi m_1 =gf_1m_1=0$. It follows that $g \in
l_{S}(m_1)$, i.e. $l_{S}(f_1) \subseteq l_{S}(m_1)$, and thus
$l_{S}(f_1)=l_{S}(m_1)$. Now let $b \in r_{R}(m_1)$. Then $m_1b=0$
and so $a_1b=\varphi m_1b=0$. Hence $b \in r_{R}(a_1)$ and so
$r_{R}(m_1)\subseteq r_{R}(a_1)$. For the reverse inclusion, let
$b \in r_{R}(a_1)$. Then $m_1b=m_1\varphi m_1b=m_1a_1b=0$ and
hence $b \in r_{R}(m_1)$. So, $r_{R}(m_1)= r_{R}(a_1)$.

Next we will prove that
$l_{S}(f) \subseteq l_{S}(f_1)$ and $r_{R}(a)\subseteq
r_{R}(a_1)$. Let $g \in l_{S}(f)$. Then $g f=0$ and so $gf_1 = g
m_1\varphi =gfm_1\varphi=0$. Hence $g \in l_{S}(f_1)$, i.e.
$l_{S}(f) \subseteq l_{S}(f_1)$. For the other inclusion, let $b
\in r_{R}(a)$. Then $ab=0$ and so $a_1b=\varphi m_1b=\varphi
m_1ab=0$. It follows that $b \in r_{R}(a_1)$, i.e. $r_{R}(a)\subseteq
r_{R}(a_1)$. Since then $(1_S-f)f_1=0$ and $a_1(1_R-a)=0$ and since $f_1$ and $a_1$ are idempotents, we have $f_1=ff_1 =
f_1 f f_1$ and $a_1=a_1 a= a_1 a a_1$. Note that $f_1 f \in S$ and $a
a_1 \in R$ are idempotents.

 We will now prove that $l_{S}(m_1)=l_{S}(f_1f)$,
 $r_{R}(m_1)=r_{R}(aa_1)$.
 Clearly, $l_{S}(m_1)=l_{S}(f_1)\subseteq l_{S}(f_1f)$ and
 $r_{R}(m_1)=r_{R}(a_1)\subseteq r_{R}(a a_1)$. For any $x\in
l_{S}(f_1f)$ and any $y\in r_{R}(a a_1)$, we have $xf_1=xf_1ff_1=0$
and $a_1y=a_1aa_1y=0$. This yields $x\in l_{S}(f_1)=l_{S}(m_1)$ and
$y\in r_{R}(a_1)= r_{R}(m_1)$. It follows that
$l_{S}(m_1)=l_{S}(f_1f)$ and $r_{R}(m_1)=r_{R}(aa_1)$, as desired.

Finally, from $fm_1=fm_2$ and $m_1a=m_2a$, we establish
that $f_1fm_1=f_1fm_2$ and $m_1aa_1=m_2aa_1$. Therefore
$m_1\leq^{-}m_2$.

The converse implication follows directly by Theorem \ref{main}.
\end{proof}

\begin{cor} Let $M$ be a regular module and $m_1, m_2\in M$.
Then $m_1\leq^{-} m_2$ if and only if there exist $f^2=f\in S$,
$a^2=a\in R$ such that $m_1R \subseteq fM$,  $Sm_1 \subseteq Ma$,
$fm_1=fm_2$, and $m_1a=m_2a$.
\end{cor}
\begin{proof} Note that for any $m\in M$, $f^2=f\in S$ and $a^2=a\in R$,
\begin{enumerate}
    \item $l_S(f) \subseteq l_S(m)$ if and only if $mR \subseteq
    fM$,
    \item $r_R(a) \subseteq r_R(m)$ if and only if $Sm \subseteq
    Ma$.
\end{enumerate}
The result is thus an immediate consequence of Theorem \ref{elif}.
\end{proof}

There are many definitions that are equivalent to Hartwig's
definition of the minus partial order. One of the most commonly
used is the following which is due to Jones (see for example
\cite{Higgins}). Let $\mathcal{S}$ be a semigroup. Then
\begin{equation}
a\leq_{J}b\text{\quad if\quad}a=pb=bq \label{Jones}
\end{equation}
where $p,q$ are idempotent elements in $\mathcal{S}^{1}$ and
$\mathcal{S}^{1}$ denotes $\mathcal{S}$ if $\mathcal{S}$ has the
identity, and $\mathcal{S}$ with the identity adjoined in the
other case. Jones introduced this relation for the setting of an
arbitrary semigroup. If $\mathcal{S}$ is a regular semigroup, then
it is known (see \cite{Mitsch}) that $\leq_{J}$ is a partial order
on $\mathcal{S}$. We will now introduce a new relation on a module $M$
that is analogous to (\ref{Jones}) and then show that when $M$ is
a regular module this new relation is the minus partial order on
$M$.

\begin{defn}
Let $M$ be a module and $m_1,m_2 \in M$. We write $m_1\leq_{J}m_2$
if there exist idempotent elements $f \in S$ and $a \in R$ such
that $m_1=fm_2=m_2 a$.
\end{defn}

\begin{lem}\label{sait} Let $M$ be a module and $m_1, m_2\in M$. If $m_1\leq^{-} m_2$,
then  there exist $f^2=f\in S$, $a^2=a\in R$ such that
$m_1=fm_1=fm_2=m_1a=m_2a$.
\end{lem}
\begin{proof} If $m_1\leq^{-} m_2$, then there exist $f^2=f\in S$, $a^2=a\in R$ such that
$l_S(m_1)=l_S(f)$, $r_R(m_1)=r_R(a)$, $fm_1=fm_2$, and $m_1a=m_2a$.
Hence $(1_S-f)m_1=0$ and $m_1(1_R-a)=0$, and therefore
$m_1=fm_1=fm_2=m_1a=m_2a$.
\end{proof}

\begin{thm}
\label{equiv-1} Let $M$ be a regular module and $m_1, m_2\in M$.
Then $m_1\leq_{J}m_2$ if and only if $m_1\leq^{-} m_2$.
\end{thm}

\begin{proof}
If $m_1\leq^{-} m_2$, then $m_1\leq_{J}m_2$ by Lemma \ref{sait}. Conversely,
let $m_1\leq_{J}m_2$. Then $m_1=fm_2=m_2a$ for some $f^{2}=f \in
S$, $a^{2}=a\in R$. Clearly, $fm_1=fm_2$ and $m_1a=m_2a$. From
$m_1 =fm_2$ we get $l_{S}(f)\subseteq l_{S}(m_1)$ and by
$m_1=m_2a$ we obtain $r_{R }(a)\subseteq r_{R}(m_1)$. Therefore an
application of Theorem \ref{elif} completes the proof.
\end{proof}

In \cite{Mitsch}, Mitsch generalized Hartwig's definition of the
minus partial order to arbitrary semigroups. The definition
follows: Suppose $a,b$ are two elements of an arbitrary semigroup
$\mathcal{S}$. Then we write
\begin{equation}
a\leq_{M}b\text{\quad if\quad}a=xb=by \text{ and }xa=a \label{MitschDefMinus}%
\end{equation}
for some elements $x,y\in\mathcal{S}^{1}$. Mitsch proved that
$\leq_{M}$ is indeed a partial order for any semigroup
$\mathcal{S}$ and that for $a,b\in\mathcal{S}$,
\begin{equation} \label{MitschMinusEquiv}
a\leq_{M}b\text{\quad if and only if\quad}a=xb=by\text{ and
}xa=ay=a
\end{equation}
for some elements $x,y\in\mathcal{S}^{1}$. Following the above
ideas, we now introduce a new relation $\leq_{M}$ on  a module $M$.

\begin{defn}
\label{Definition_Mitsch_power_set} Let $M$ be a module and
$m_1,m_2 \in M$. We write $m_1\leq _{M}m_2$ if there exist $f \in
S$ and $a \in R$ such that $m_1 =m_2a=fm_2$ and $m_1=fm_1$.
\end{defn}

We will show that when $M$ is a regular module, $\leq_{M}$ is again
the minus partial order on $M$. First, we present an auxiliary result which was inspired by (\ref{MitschMinusEquiv}).

\begin{lem}\label{brc} Let $M$ be a module and $m_1,m_2 \in M$.
Then $m_1\leq _{M}m_2$ if and only if there exist $f \in S$ and $a
\in R$ such that
$${m_1 =m_2a=fm_2 \quad \text{and} \quad m_1=fm_1=m_1a.}$$
\end{lem}
\begin{proof} Let  $m_1\leq _{M}m_2$. Then there exist $f \in
S$ and $a \in R$ such that $m_1 =m_2a=fm_2$ and $m_1=fm_1$. Hence
$m_1a=(fm_2)a=f(m_2a)=fm_1=m_1$. The converse is obvious.
\end{proof}

\begin{thm} \label{Mitsch_thm}
Let $M$ be a regular module and $m_1,m_2 \in M$. Then
$m_1\leq^{-}m_2$ if and only if $m_1\leq_{M}m_2$.
\end{thm}

\begin{proof}
Let $m_1\leq^{-}m_2$. Then by Theorem \ref{equiv-1} there exist
$f^2=f\in S$ and $a^2=a\in R$ with $m_1=fm_2=m_2a$. It follows,
$(1_S-f)m_1=0$ and so $m_1=fm_1$. Hence $m_1\leq _{M}m_2$.

Conversely, let $m_1\leq_{M}m_2$. Then there exist $f\in S$ and
$a\in R$ with $m_1=fm_2=m_2a$ and $m_1=fm_1$. Since $M$ is a
regular module, there exists $\varphi \in M^*$ such that
$m_1=m_1\varphi m_1$. Let $g=m_1\varphi \in S$ and $b=\varphi
m_1\in R$. Obviously, $g^2=g$ and $b^2=b$. Also, clearly,
$l_S(m_1)=l_S(g)$ and $r_R(m_1)=r_R(b)$. Equation $m_1=fm_1$
implies $(1_S-f)m_1=0$. Therefore $(1_S-f)g=0$ and
hence $g=fg$. It follows $(gf)^2=gf$. Also, $m_1a=m_1$ by Lemma
\ref{brc}. Then $m_1(1_R-a)=0$ and so $b(1_R-a)=0$
which implies $b=ba$. Hence $(ab)^2=ab$. We now take into
account the following items:

\begin{enumerate}
\item[(i)] Let $x\in l_S(gf)$. Then $0=xgf=xm_1\varphi f$ and so $$0=xm_1\varphi fm_1=xm_1\varphi
m_1=xm_1.$$ Hence $x\in l_S(m_1)$ and thus $l_S(gf)\subseteq
l_S(m_1)$.

\item[(ii)] Let $x\in r_R(ab)$. Then $0=abx=a\phi m_1x$ and so $$0=m_2a\varphi m_1x=m_1\varphi
m_1x=m_1x.$$ Hence $x\in r_R(m_1)$ and thus $r_R(ab)\subseteq
r_R(m_1)$.

\item[(iii)] $gfm_1=gm_1=gfm_2$.

\item[(iv)] $m_1ab=m_1b=m_2ab$.
\end{enumerate}
By Theorem \ref{elif}, we may conclude that $m_1\leq^{-}m_2$.
\end{proof}

\begin{cor} Let $M$ be a regular module and $m_1,m_2 \in M$.
The following statements are then equivalent:
\begin{enumerate}
\item[{\rm(i)}] $m_1\leq^{-}m_2$;
\item[{\rm(ii)}] There exist $g \in S$ and $b \in R$ such that $m_1 =m_2b=gm_2$ and
$m_1=m_1b$.
\end{enumerate}
\end{cor}
\begin{proof}(i) $\Rightarrow$  (ii) Clear by Theorem \ref{Mitsch_thm} and Lemma \ref{brc}. \\
(ii) $\Rightarrow$ (i) There exist $g \in S$ and $b \in R$ such
that $m_1 =m_2b=gm_2$ and $m_1=m_1b$. We have
$$gm_1=g(m_2b)=(gm_2)b=m_1b=m_1$$ and thus (i) holds.
\end{proof}

With the next result, which follows directly by Lemma \ref{sait}, we will present a relation between annihilators of
elements $m_1$ and $m_2$ of a module $M$ where $m_1$ and $m_2$ are connected by the
minus partial order.

\begin{prop}\label{isil} Let $M$ be a  module and $m_1, m_2 \in M$. If
$m_1\leq^{-}m_2$, then $l_{S}(m_2)\subseteq l_{S}(m_1)$ and
$r_{R}(m_2)\subseteq r_{R}(m_1)$.
\end{prop}

The following example shows that the converse statement of
Proposition \ref{isil}  need not hold in general.

\begin{ex}{\rm
Let $\mathbb{Z}_{10}$ denote the ring of integers modulo $10$. It
is known that End$_{\mathbb Z}(\mathbb Z_{10})\cong \mathbb Z_{10}$. Since
$\mathbb{Z}_{10}$ is a semisimple $\mathbb{Z}$-module, it is a
regular module. Note that the nontrivial idempotents of
$\mathbb{Z}_{10}$ are $\overline{5}$ and $\overline{6}$. Consider
the elements $m_1 =\overline{2}$ and $m_2=\overline{6}$ of
$\mathbb{Z}_{10}$.
Then, on the one hand, $l_{_{\mathbb{Z}_{10}}}(m_1)=l_{_{\mathbb{Z}%
_{10}}}(m_2)=\{\overline{0},\overline{5}\}$ and
$r_{_{\mathbb{Z}}}(m_1)=r_{_{\mathbb{Z}}}(m_2)=5\mathbb Z$. On the
other hand, there are no idempotents $f \in$ End$_{\mathbb Z}(\mathbb
Z_{10})$ and $a \in \mathbb{Z}$ such that $m_1 =fm_2=m_2a$. Thus,
$m_1\not \leq ^{-}m_2$. }
\end{ex}

In \cite{BJPS}, Blackwood at al. defined a partial order
$\leq^{\oplus}$ in the following way: $$a \leq^{\oplus} b\quad \text{if} \quad
bR = aR \oplus (b-a)R$$ and called it  \textit{the direct sum partial
order}. The direct sum partial order is known to be
equivalent to the minus partial order on von Neumann regular
rings (see \cite[Lemma 3]{BJPS}). Motivated by the notion of the direct sum partial order on (regular) rings, we will conclude the paper by introducing a new relation on a module $M$ and then show that this relation is the minus partial order when the module $M$ is regular.

\begin{defn}\label{DirectSumDef} Let $M$ be a module and $m_1, m_2\in M$.
We write $$m_1\leq^{\oplus} m_2 \quad if \quad m_2 R = m_1 R \oplus
(m_2 -m_1)R.$$
\end{defn}

\begin{thm}\label{hoca} Let $M$ be a module and $m_1, m_2\in M$ with $m_1$ and $m_2$ regular.
 Then the following statements are equivalent:
\begin{enumerate}
\item[{\rm(i)}] $m_1\leq^{-} m_2$;
\item[{\rm(ii)}] $m_1\leq^{\oplus} m_2$.
\end{enumerate}
\end{thm}
\begin{proof} $(1) \Rightarrow (2)$: Let $m_1\leq^{-} m_2$.
There exists $\varphi\in M^*$ such that $m_1 =m_1\varphi
m_1$. By Remark \ref{burcu}, we have $M=m_1 R \oplus T$ where
$T=\{x \in M ~|~m_1\varphi x=0\}$. To prove $m_1\leq^{\oplus} m_2$, we will show that $m_2 R=m_1 R \oplus
(m_2 -m_1)R$. Let first $x\in m_2R$. Then $x=m_2r$ for some $r\in R$ and therefore
$$x=m_2r+m_1r-m_1r=m_1r+(m_2-m_1)r\in m_1 R +
(m_2 -m_1)R.$$
By Proposition \ref{subset}, $m_1R\subseteq m_2R$ which yields $(m_2-m_1)R\subseteq m_2R$. Thus,
$$m_1R+(m_2-m_1)R\subseteq m_2R.$$ Let $a\in R.$ Observe that $$m_1\varphi (m_2-m_1)a= m_1\varphi m_2 a - m_1\varphi m_1a =m_1\varphi m_1a - m_1\varphi m_1a =0$$ and therefore
$(m_2-m_1)R \subseteq T$. Since $M=m_1R\oplus T$, it follows that $m_1R\cap (m_2-m_1)R=\{0\}.$ Thus, $m_2 R=m_1 R \oplus
(m_2 -m_1)R$.

\noindent (2) $\Rightarrow$ (1): Let $m_1\leq^{\oplus}
m_2$, i.e. $m_2R= m_1R\oplus (m_2 - m_1)R$. There exist $\varphi_2 \in M^*$ and $r\in R$
with $m_2 = m_2\varphi_2m_2$ and $m_1 = m_2r$.
Then $m_2\varphi_2m_1 = m_2\varphi_2m_2r = m_2r = m_1$ and hence
$$m_1\varphi_2m_1 = m_2\varphi_2m_1 - (m_2 - m_1)\varphi_2m_1 = m_1
- (m_2 - m_1)\varphi_2m_1.$$ Thus $m_1\varphi_2m_1 - m_1 = -(m_2 -
m_1)\varphi_2m_1\in m_1R\cap (m_2 - m_1)R$. Since $m_1R\cap (m_2 -
m_1)R = \{0\}$, it follows $m_1\varphi_2m_1 = m_1$. Therefore $m_2\varphi_2m_1 =
m_1 = m_1\varphi_2m_1$. Let $x = \varphi_2m_1\varphi_2\in M^*$.
Then $m_1xm_1 = m_1$. Note that $m_2x = m_2\varphi_2m_1\varphi_2 =
m_1\varphi_2$ and $m_1x = m_1\varphi_2m_1\varphi_2 =
m_1\varphi_2$ and hence $m_1x = m_2x$. Also,
$m_2\varphi_2(m_2 - m_1) = m_2\varphi_2m_2  - m_2\varphi_2m_1 =
m_2 - m_1$, so  $$\begin{array}{lll}(m_2 - m_1)\varphi_2(m_2 -
m_1) &=& m_2\varphi_2(m_2 - m_1)- m_1\varphi_2(m_2 - m_1) \\ &=&
(m_2 - m_1) - m_1 \varphi_2(m_2 - m_1)\end{array}$$ and this
implies that $m_1\varphi_2(m_2 - m_1) = (m_2 - m_1) - (m_2 - m_1)
\varphi_2(m_2 - m_1)\in m_1R\cap (m_2 - m_1)R$. Since $m_1R\cap
(m_2 - m_1)R = 0$, we obtain $m_1\varphi_2m_2 = m_1\varphi_2m_1$ and $m_2 -
m_1 = (m_2 - m_1)\varphi_2(m_2 - m_1)$. From $m_1\varphi_2m_1 =
m_1$, it follows $m_1\varphi_2m_2 = m_1$. Now, $xm_2 =
\varphi_2m_1\varphi_2m_2 = \varphi_2m_1$ and $xm_1 =
\varphi_2m_1\varphi_2m_1 = \varphi_2m_1$. Thus $xm_1 = xm_2$ and
therefore $m_1\leq^{-} m_2$.
\end{proof}

The following result follows directly from Theorem \ref{hoca}.

\begin{cor} If $M$ is a regular module, then Definitions \ref{ilk} and \ref{DirectSumDef} are equivalent, i.e. the relation
$\leq^{\oplus}$ is the minus partial order on $M$.
\end{cor}

It can be seen that for nonzero elements, the minus order is the
identity relation on vector spaces. But this is not the case if
the module is regular as the following example shows.


\begin{ex}{\rm Consider $M=\Bbb Z_6$ as an $R=\Bbb Z_{30}$-module.
On the one hand, since $R\cong M\oplus \Bbb Z_5$, $M$ is a
projective $R$-module. On the other hand, $M$ is semisimple as an
$R$-module. Hence $M$ is a regular $R$-module. Let
$m_1=\overline{2}$, $m_2=\overline{5}\in M$. Then $m_2R=m_1R\oplus
(m_2-m_1)R$, i.e., $\overline{5}\Bbb Z_{30}=M=\overline{2}\Bbb
Z_{30}\oplus \overline{3}\Bbb Z_{30}$. This implies
$m_1\leq^{\oplus} m_2$. Therefore $m_1\leq^{-} m_2$ by Theorem
\ref{hoca}.}
\end{ex}

\noindent {\bf Acknowledgment.} The work is partially supported by
Slovene Research Agency, ARRS (Slovene-Turkish Grant
BI-TR/17-19-004) and by The Scientific and Technological Research
Council of Turkey, TUBITAK (Grant TUBITAK-116F435). The authors
wish to thank ARRS and TUBITAK for financial support.

\end{document}